\documentclass[reqno]{amsart}

\usepackage[utf8]{inputenc}
\usepackage[T1]{fontenc}
\usepackage{amsthm}
\usepackage{amsmath}
\usepackage{amssymb}
\usepackage[inline]{enumitem}
\usepackage[dvips]{graphicx}
\usepackage{color}
\usepackage{comment}
\usepackage{hyperref}
\usepackage{url}
\usepackage{stackrel}
\usepackage[left=3.5cm, right=3.5cm, paperheight=11.4in]{geometry}
\usepackage{fancyhdr}
\usepackage{mathrsfs}
\usepackage{stmaryrd}
\usepackage{nicefrac}
\usepackage{soul}
\usepackage{datetime}
\longdate

\AtBeginDocument{%
   \def\MR#1{}
}

\newtheorem{theorem}{Theorem}
\newtheorem{prop}[theorem]{Proposition}
\newtheorem{corollary}[theorem]{Corollary}

\theoremstyle{definition}

\newtheorem*{def*}{Definition}

\newtheorem{remark}[theorem]{\textbf{Remark}}

\theoremstyle{remark}
\newtheorem{claim}{\textsc{Claim}}
\newtheorem*{claim*}{\textsc{Claim}}

\hypersetup{
    pdftitle={On the iterates of shifted Euler's function},
    pdfauthor={Paolo Leonetti and Florian Luca},
    pdfmenubar=false,
    pdffitwindow=true,
    pdfstartview=FitH,
    colorlinks=true,
    linkcolor=blue,
    citecolor=green,
    urlcolor=cyan
}

\DeclareMathSymbol{\widehatsym}{\mathord}{largesymbols}{"62}

\renewcommand{\rho}{\varrho}

\providecommand{\NNb}{\mathbb{N}}

\begin{document}
\title{On the iterates of shifted Euler's function}

\author{Paolo Leonetti}
\address{Department of Economics, Universit\`a degli Studi dell'Insubria, via Monte Generoso 71, 21100 Varese, Italy}
\email{leonetti.paolo@gmail.com}

\author{Florian Luca}
\address{School of Mathematics, University of the Witwatersrand, Private Bag 3, Wits 2050, South Africa}
\email{Florian.Luca@wits.ac.za}


\begin{abstract}
\noindent{} Let $\varphi$ be the Euler's function and fix an integer $k\ge 0$. We show that, for every initial value $x_1\ge 1$, the sequence of positive integers $(x_n)_{n\ge 1}$ defined by $x_{n+1}=\varphi(x_n)+k$ for all $n\ge 1$ is eventually periodic. Similarly, for every initial value $x_1,x_2\ge 1$, the sequence of positive integers $(x_n)_{n\ge 1}$ defined by $x_{n+2}=\varphi(x_{n+1})+\varphi(x_n)+k$ for all $n\ge 1$ is eventually periodic, provided that $k$ is even. 
\end{abstract}
\maketitle
\thispagestyle{empty}
%
\section{Introduction and Main Results}\label{sec:introduction}

Let $\mathbb{N}$ be the set of positive integers and fix an arithmetic function $f: \mathbb{N}^d\to \mathbb{N}$, for some $d \in \mathbb{N}$. Let $(x_n)_{n\ge 1}$ be a sequence of positive integers which satisfies the recurrence 
\begin{equation}\label{eq:defrecurrence}
x_{n+d}=f(x_n,\ldots,x_{n+d-2},x_{n+d-1}) 
\quad \text{ for all }n \in \mathbb{N},
\end{equation}
with starting values $x_1,\ldots,x_d\in \mathbb{N}$. In the case where $d=1$ the sequence $(x_n)_{n\ge 1}$ is simply the orbit of $x_1$ with respect to $f$. The aim of this note is to study whether certain recurrence sequences $(x_n)_{n\ge 1}$ of the type \eqref{eq:defrecurrence} are eventually periodic independently of their starting values, that is, for all $x_1,\ldots,x_d \in \mathbb{N}$ there exists $T \in \mathbb{N}$ such that $x_n=x_{n+T}$ for all sufficiently large $n$. 

Here and after, we will use implicitly the basic observation that a recurrence sequence $(x_n)_{n\ge 1}$ as in \eqref{eq:defrecurrence} is eventually periodic if and only if it is bounded, cf. \cite[p.~45]{GEverest}. 

We start with a simple result for functions $f$ which are not too large: 
\begin{prop}\label{prop:basic}
Let $f: \mathbb{N}^d \to \mathbb{N}$ be a 
arithmetical 
function, with $d \in \mathbb{N}$, and 
suppose 
that there exists $C \in \mathbb{N}$ such that 
\begin{equation}\label{eq:decreasing}
f(n_1,\ldots,n_d) < \max\{n_1,\ldots,n_d\}
\end{equation}
for all $n_1,\ldots,n_d \in \mathbb{N}$ with $n_i \ge C$ for some $i\in \{1,\ldots,d\}$. 
Let $(x_n)_{n\ge 1}$ be a recurrence sequence as in \eqref{eq:defrecurrence}, with starting values $x_1,\ldots,x_d \in \mathbb{N}$.
%
Then 
\begin{equation}\label{eq:constant1}
\limsup_{n\to \infty} x_n \le \max \{f(n_1,\ldots,n_d): n_1,\ldots,n_d\le C-1\}. 
\end{equation}
In particular, 
$(x_n)_{n\ge 1}$ 
is upper bounded; hence, eventually periodic. 
\end{prop}

Special instances of Proposition \ref{prop:basic} in the one-dimensional case $d=1$ have been previously obtained in the literature. For example, Porges \cite{MR12611} considered the case where $f(n)$ is the sum of squares of the digits of $n$, 
cf. also 
\cite{MR40560, MR4390667, MR122751}. 
Note that condition \eqref{eq:decreasing} holds if $f(n_1,\ldots,n_d)=o(N)$ as $N=\max\{n_1,\ldots,n_d\}\to \infty$. 

Of course, there exist other functions $f$ which do not satisfy condition \eqref{eq:decreasing} and such that every sequence $(x_n)_{n\ge 1}$ as in \eqref{eq:defrecurrence} is eventually periodic: as a trivial example, one can consider $f(n_1,\ldots,n_d):=\max\{n_1,\ldots,n_d\}$ for all $(n_1,\ldots,n_d) \in \mathbb{N}^d$. On the opposite extreme, if $f$ is slightly bigger (for instance, $f(n_1,\ldots,n_d):=\max\{n_1,\ldots,n_d\}+1$) then there are no eventually periodic sequences $(x_n)_{n\ge 1}$ as in \eqref{eq:defrecurrence}. This is the starting point for this work, which motivates the following heuristic: \emph{if a function $f$ satisfies \eqref{eq:decreasing} \textquotedblleft on average,\textquotedblright\ then every sequence $(x_{n})_{n\ge 1}$ as \eqref{eq:defrecurrence} should be eventually periodic, independently of its starting values.} 

We are going to confirm the above heuristic in two cases which involve (shifted iterates of) the Euler's function $\varphi$ (recall that $\varphi(n)$ is the number of integers in $\{1,\ldots,n\}$ which are coprime with $n$). Our first main result follows.
\begin{theorem}\label{thm:main1}
Fix an integer $k\ge 0$ and let $(x_n)_{n\ge 1}$ be the recurrence sequence defined by 
$$
x_{n+1}=\varphi(x_n)+k
\quad \text{ for all }n \in \mathbb{N},
$$
with starting value $x_1 \in \mathbb{N}$. Then 
\begin{equation}\label{eq:eq:claimthm1}
\sup_{n \in \mathbb{N}} 
x_n \le  \max\{x_1,k^4\}+(k+1)^2. 
\end{equation}
In particular, $(x_n)_{n\ge 1}$ is 
eventually periodic. 
\end{theorem}

Note that the dependence of the upper bound \eqref{eq:eq:claimthm1} on $x_1$ cannot be removed: indeed, if $k:=x_1-\varphi(x_1)$ for some $x_1\in \mathbb{N}$, then the sequence $(x_{n})_{n\ge 1}$ is constantly equal to $x_1$. 

In addition, the trivial case $k=0$ in Theorem \ref{thm:main1} has been already considered in the literature from different viewpoints (and, of course, it follows by Proposition \ref{prop:basic} since $\varphi(n)\le n-1$ for all $n\ge 2$). Indeed, given a starting value $x_1 \in \mathbb{N}$, then $x_{n+1}=\varphi^{(n)}(x_1)$ for all $n \in \mathbb{N}$, where $\varphi^{(m)}$ is the $m$-th fold iteration of $\varphi$. For instance, Pillai \cite{MR1561820} showed that
$$
\left\lfloor \frac{\log x_1-\log 2}{\log 3}\right\rfloor+1 \le N(x_1) \le 
\left\lfloor \frac{\log x_1}{\log 2}\right\rfloor+1
\quad 
\text{ for all }x_1 \in \mathbb{N},
$$
where $N(x_1)$ is the minimal integer $n$ for which $x_n=1$ (see also \cite{MR7755}) and it has been conjectured by Erd{\H o}s et al. \cite{MR1084181} that $N(x_1) \sim \alpha \log x_1$ as $x_1\to \infty$, for some $\alpha \in \mathbb{R}$.  
It is known that the understanding of the multiplicative structure of $\varphi$ and its iterates is in some sense equivalent to the study of the behavior of the integers of the form $p-1$, where $p$ is prime. See also \cite{MR2289421, MR2393364, MR2860571, MR2193156} for related works.

On the other hand, if $k\ge 1$ then the function $f(n):=\varphi(n)+k$ does not satisfy condition \eqref{eq:decreasing}: indeed, $\varphi(p)=p-1$ for all primes $p$, hence $f(p)\ge p$. However, it is well known that 
\begin{equation}\label{eq:approximation}
\frac{1}{n}\sum_{i=1}^n \varphi(i)=\frac{3}{\pi^2}n+O(\log n) 
\quad \text{ as }n\to \infty
\end{equation}
(see e.g. 
\cite[Theorem 8.6]{MR2919246}). Hence, approximating roughly $f(n)$ with $cn+k$, where $c=3/\pi^2 \in (0,1)$, we expect that \eqref{eq:decreasing} holds \textquotedblleft on average,\textquotedblright\ which is the heuristic behind Theorem \ref{thm:main1}. 

Our second main result is as follows. 
\begin{theorem}\label{thm:main2}
Fix an even integer $k\ge 0$ and let $(x_n)_{n\ge 1}$ be the recurrence sequence defined by 
$$
x_{n+2}=\varphi(x_{n+1})+\varphi(x_n)+k
\quad \text{ for all }n \in \mathbb{N},
$$
with starting values $x_1,x_2 \in \mathbb{N}$. Then 
\begin{equation}\label{eq:eq:claimthm2}
\sup_{n\in \mathbb{N}} x_n \le 
4^{X^{3^{k+1}}}, 
\quad \text{ where }\,\,
X:=\frac{3x_1+5x_2+7k}{2}.
\end{equation}
In particular, $(x_n)_{n\ge 1}$ is 
eventually periodic. 
\end{theorem}
The heuristic supporting Theorem \ref{thm:main2} is similar: thanks to \eqref{eq:approximation}, the value $f(n,m):=\varphi(n)+\varphi(m)+k$ can be roughly upper bounded by $2c\max\{n,m\}+k$, which is definitively smaller than $\max\{n,m\}$ since $2c=6/\pi^2 <1$. 

We end with an open question to check whether, if $d$ is sufficiently large, then there exist starting values $x_1,\ldots,x_d \in \mathbb{N}$ such that the sequence $(x_n)_{n\ge 1}$ defined as in \eqref{eq:defrecurrence} with 
$$
f(n_1,\ldots,n_d)=\varphi(n_1)+\cdots+\varphi(n_d)
$$
is \emph{not} eventually periodic. In a sense, this is somehow related to the open question known as Lehmer's totient problem \cite{Lehmer}, which asks about the existence of a composite $q\ge 2$ such that $\varphi(q)$ divides $q-1$: indeed, if $r:=(q-1)/\varphi(q)$, $d=r$, and $x_1=\cdots=x_d=q$, then the sequence $(x_n)_{n\ge 1}$ would be constant.

\section{Proofs}

We start with the proof of Proposition \ref{prop:basic}.

\begin{proof}
[Proof of Proposition \ref{prop:basic}]
Let $\mathcal{Q}$ be the set of $d$-uples $(n_1,\ldots,n_d) \in \mathbb{N}^d$ with $n_i\le C-1$ for all $i \in\{1,\ldots,d\}$. 
Suppose that $(n_1,\ldots,n_d) \in  \mathbb{N}\setminus \mathcal{Q}$. 
Hence, we can pick the largest index $i \in \{1,\ldots,d\}$ such that $n_i=\max\{n_1,\ldots,n_d\}$. In particular, $n_i\ge C$. 
We claim that there exists $m \in \mathbb{N}$ such that $(n_{m+1}, \ldots, n_{m+d}) \in \mathcal{Q}$. In addition, if $m$ is the least such integer then $\max\{n_j,\ldots,n_{j+d-1}\}$ is decreasing for $j \in \{1,\ldots,m\}$. 
To this aim, suppose for the sake of contradiction that the claim does not hold. 
By the standing hypothesis  \eqref{eq:decreasing}, we get $n_{d+1}=f(n_1,\ldots,n_d)<  n_{i}$. Repeating this reasoning, we obtain that $n_{d+j}< \max\{n_j,\ldots,n_{d+j-1}\}= n_{i}$ for all $j\in \{1,\ldots,i\}$. Hence, $\max\{n_{i+1},\ldots,n_{i+d}\}\le n_{i}-1$. 
Proceeding similarly, it follows that 
$$
\max\{n_{i+(k-1)d+1},\ldots,n_{i+kd}\}\le n_{i}-k 
\quad \text{ for all }k\in \mathbb{N}. 
$$
However, if $k=n_{i}+1-C$ then $(n_{i+(k-1)d+1},\ldots,n_{i+kd}) \in \mathcal{Q}$, which proves the claim. 

To complete the proof, fix starting values $x_1,\ldots,x_d \in \mathbb{N}$. 
By the above claim and the finiteness of $\mathcal{Q}$, 
it follows that the sequence $(x_n)_{n\ge 1}$ is upper bounded by the constant 
$$
\max\left(\{x_1,\ldots,x_d\}\cup \{f(n_1,\ldots,n_d): (n_1,\ldots,n_d) \in \mathcal{Q}\}\right)
$$
and that the upper limit in \eqref{eq:constant1} holds.
\end{proof}

\begin{proof}
[Proof of Theorem \ref{thm:main1}]
First, let us suppose $k \le 1$ and fix a starting value $x_1 \in \mathbb{N}$. Then 
$$
x_{n+1}=\varphi(x_n)+k \le \max\{1,x_n-1\}+1=\max\{2,x_n\} 
$$
for all $n \in \mathbb{N}$, 
with the consequence that $x_n \le \max\{x_1,2\}$  for all $n \in \mathbb{N}$. 

Suppose hereafter that $k \ge 2$. Note that, for all $n,m \in \mathbb{N}$, we have
\begin{equation}\label{eq:trivialbound0}
\begin{split}
x_{n+m}&\le \max\{x_{n+m-1}-1,1\}+k \\
&= \max\{x_{n+m-1}+k-1,k+1\} \le \max\{x_n+m(k-1),k+1\}.
\end{split}
\end{equation}

Let us suppose for the sake of contradiction that $(x_n)_{n\ge 1}$ is not upper bounded. Hence, there exists a minimal $r_1 \in \mathbb{N}$ such that $x_{r_1} \ge k^4$ (in particular, $x_{r_1}>4$).

\begin{claim}\label{claim:exitencei}
There exists $i \in \{1,\ldots,k\}$ such that $x_{r_1+i}<x_{r_1}$.
\end{claim}
\begin{proof}
Since $\varphi(n) \le n-\sqrt{n}$ whenever $n$ is composite (by the fact that there exists a divisor of $n$ which is at most $\sqrt{n}$), it follows that, if $x_{{r_1}+i-1}$ is composite for some $i \in \{1,\ldots,k\}$, then
\begin{displaymath}
x_{r_1+i}=\varphi(x_{r_1+i-1})+k \le x_{r_1+i-1}-\sqrt{x_{r_1+i-1}}+k.
\end{displaymath}
Considering that the map $x\mapsto x-\sqrt{x}$ is increasing on $(4,\infty)$ and using \eqref{eq:trivialbound0}, we obtain
\begin{displaymath}
\begin{split}
x_{r_1+i} &\le x_{r_1}+(i-1)(k-1)-\sqrt{x_{r_1}+(i-1)(k-1)}+k\\
&\le x_{r_1}+(k-1)^2-\sqrt{k^4}+k < x_{r_1}.
\end{split}
\end{displaymath}

To conclude, we show that there exists some $i \in \{1,\ldots,k\}$ for which $x_{{r_1}+i-1}$ is composite. Indeed, in the opposite, these $x_{{r_1}+i-1}$s are all primes (and greater than $k$), hence
$$
x_{r_1+i-1}=x_{r_1}+(i-1)(k-1) \equiv x_{r_1}-(i-1) \bmod{k}
$$
for all $i \in \{1,\ldots,k\}$. 
This is impossible, because there would exist $i \in \{1,\ldots,k\}$ such that $k$ divides $x_{{r_1}+i-1}$. 
\end{proof}

At this point, inequality \eqref{eq:trivialbound0} and Claim \ref{claim:exitencei} imply that there exists a minimal $i_1 \in \{1,\ldots,k\}$ such that $x_{r_1+i_1}< x_{r_1}$; hence
\begin{equation}\label{eq1}
\max\{x_1,\ldots,x_{r_1+i_1-1}\} \le x_{r_1}+(i_1-1)(k-1) < x_{r_1}+k^2.
\end{equation}

With the same reasoning, we can construct recursively sequences of positive integers $(r_n)$ and $(i_n)$ such that for all $n \in \NNb$: 
\begin{enumerate}[label={\rm (\roman{*})}]
\item $r_{n+1}$ is the minimal integer such that $r_{n+1}\ge r_n+i_n$ and $x_{r_{n+1}}\ge k^4$;
\item $i_{n+1}$ is the minimal integer in $\{1,\ldots,k\}$ such that $x_{r_{n+1}+i_{n+1}}<x_{r_n}$; hence, 
\begin{equation}\label{eq2}
\max\{x_{r_n+i_n},\ldots,x_{r_{n+1}+i_{n+1}-1}\} \le x_{r_{n+1}}+(i_{n+1}-1)(k-1) < x_{r_{n+1}}+k^2.
\end{equation}
\end{enumerate}

Lastly, note that 
\begin{equation}\label{eq3}
x_{r_{n}}=\varphi(x_{r_n-1})+k\le x_{r_n-1}+k \le k^4+k\le x_{r_1}+k
\end{equation}
for all $n \in \mathbb{N}$. 
Using inequalities \eqref{eq1}, \eqref{eq2}, and \eqref{eq3}, we conclude that 
$$
x_n < x_{r_1}+k^2+k\le \max\{x_1,k^4+k\}+k^2+k<\max\{x_1,k^4\}+(k+1)^2
$$
for all $n \in \mathbb{N}$. 
This proves \eqref{eq:eq:claimthm1}, concluding the proof. 
\end{proof}

\begin{remark}
A sketch of a shorter proof that the sequence $(x_n)_{n\ge 1}$ in Theorem \ref{thm:main1} is eventually periodic goes as follows: Set $y_n:=\varphi(x_n)$ for all $n \in \mathbb{N}$ and note that $y_n\le C n$ with $C:=\max\{x_1,k\}$. Hence $\{y_1,\ldots,y_n\}$ is contained in $V_n:=\varphi(\mathbb{N})\cap [1,Cn]$. By a classical result of Pillai \cite{MR1561819}, cf. also \cite{MR1642874} and references therein, we have $|V_n|=o(n)$ as $n\to \infty$, hence there exist distinct $i,j \in \mathbb{N}$ with $y_i=y_j$. This implies that $x_{i+1}=x_{j+1}$, therefore  $(x_n)_{n\ge 1}$ is eventually periodic. However, this does not lead to an effective upper bound as in \eqref{eq:eq:claimthm1}. 
\end{remark}

In the proof of Theorem \ref{thm:main2}, we will need also the effective version of the third Mertens' theorem given by Rosser and Schoenfeld \cite{MR137689} in 1962 (see also \cite{MR3813836, MR3434886}). As usual, hereafter we reserve the letter $p$ for primes. 
\begin{prop}\label{prop:thirdmertens}
Let $\gamma:=\lim_n\left(\sum_{i\le n}1/i-\log n\right)=0.57721...$ be the Euler--Mascheroni's constant. Then the following inequality holds for all $x\ge 2$\textup{:}
$$
\frac{e^{-\gamma}}{\log x} \left(1- \frac{1}{\log^2 x} \right) < 
\prod_{p\le x}
\left( 1 - \frac{1}{p} \right) < \dfrac{e^{-\gamma}}{\log x} \left(1+ \frac{1}{2\log^2 x} \right).
$$ 
\end{prop}
\begin{proof}
See \cite[Theorem 7 and its Corollary]{MR137689}. 
\end{proof}

\begin{corollary}\label{cor:boundprimes}
If $x\ge 6$ then 
$$
\prod_
{x<p\le x^3}
\left( 1 - \frac{1}{p} \right)<\frac{1}{2}.
$$
\end{corollary}
\begin{proof}
Thanks to Proposition \ref{prop:thirdmertens}, for each $r \ge 2$ there exists $c_r \in (-1,\nicefrac{1}{2})$ such that 
$$
\prod_
{p\le r}
\left( 1 - \frac{1}{p} \right)
=\frac{e^{-\gamma}}{\log x} \left(1+ \frac{c_r}{\log^2 x} \right).
$$
At this point, fix $x\ge 6$. It follows that 
\begin{displaymath}
\begin{split}
\prod_
{x<p\le x^3}
\left( 1 - \frac{1}{p} \right)
&=\frac
{\prod_{p \leq x^3} \left( 1 - \frac{1}{p} \right)}
{\prod_{p \leq x} \left( 1 - \frac{1}{p} \right)}\\
&=\frac{e^{-\gamma}/(\log x^3)}{e^{-\gamma}/\log x}\cdot  \frac{(1+c_{x^3}/(\log x^3)^2)}{1+c_x/(\log x)^2}\\
& < \frac{1}{3}\cdot \frac{1+0.5/ (9\log^2 x)}{1-1/(\log^2 x)} < \frac{1}{2}.
\end{split}
\end{displaymath}
Indeed, the last inequality is equivalent to 
$$
\frac{2}{3} \left(1+\frac{1}{18 \log^2 x}\right)< 1-\frac{1}{\log^2 x},
$$
which holds if and only if 
$$
x> 
e^{\sqrt{3\left(1+\frac{1}{27}\right)}}.
$$
The conclusion follows since the value of the right hand side above is smaller than $6$. 
\end{proof}

\begin{proof}
[Proof of Theorem \ref{thm:main2}]
If $\max\{x_1,x_2\} \le 2$ and $k=0$ then 
$x_n=2$ for all $n\ge 3$, so the claimed inequality \eqref{eq:eq:claimthm2} holds since $4^{X^{3^{k+1}}}\ge 4 > x_n$ for all $n \in \mathbb{N}$. 

Hence, suppose hereafter that $\max\{x_1,x_2\} \ge 3$ or $k \ge 1$, and note that $X=\frac{1}{2}(3x_1+5x_2+7k)\ge 6$. We obtain that $\min\{x_3,x_4\}\ge 3$; hence, $x_n$ is even for all $n\ge 5$. In addition, since $\max\{x_1,\ldots,x_6\}\le 2X$ and 
$$
4^{X^{3^{k+1}}} \ge 4^X=2^{2X}>2X,
$$
it follows that the claimed inequality holds for all $n\le 6$. 

Let $(p_n)_{n\ge 1}$ be the increasing enumeration of the primes greater than $X$. 
Considering that 
$\prod_{i=1}^n\left(1-\frac{1}{p_i}\right)$ 
converges to zero as $n\to \infty$ by Proposition \ref{prop:thirdmertens}, 
one can find integers 
$1 =: r_0 < r_1 < \cdots < r_{k+1}$ such that 
\begin{equation}\label{eq:doublebounds}
\prod_{i=r_j}^{r_{j+1} - 1} \left(1 - \frac{1}{p_i}\right) < \frac{1}{2} <\prod_{i=r_j}^{r_{j+1} - 2} \left(1 - \frac{1}{p_i}\right)
\quad \text{ for all }j \in \{0, 1, \ldots, k\}.
\end{equation}
Define also 
$$
q_j:=\prod_{i=r_j}^{r_{j+1} - 1}p_i
\quad \text{ for all }j \in \{0, 1, \ldots, k\}.
$$
Since $\{q_0,q_1,\ldots,q_k\}$ are pairwise coprime, 
the Chinese remainder theorem yields the existence of some $y \in \mathbb{N}$ such that $y\equiv j\bmod{q_j}$ for all $j \in \{0,1,\ldots,k\}$. In particular
\begin{equation}\label{boundy}
y \ge q_0 \ge p_{r_0}>X.
\end{equation}

Let us suppose for the sake of contradiction that $(x_n)_{n\ge 1}$ is not bounded. Then, there exists a minimal $v \in \mathbb{N}$ such that $x_v \ge 2y$. Since $\max\{x_1,\ldots,x_6\} \le 2X$, it follows by \eqref{boundy} that $v\ge 7$. In particular, $x_{v-1}$ and $x_{v-2}$ are even.

\begin{claim}\label{claim:upperbound}
$\max\{\varphi(x_{v-1}),\varphi(x_{v-2})\} < y-k$.
\end{claim}
\begin{proof}
If $x_{v-1}<2y-2k$ then, by the fact that $x_{v-1}$ is even, $\varphi(x_{v-1}) \le \frac{1}{2}x_{v-1}<y-k$. 
Otherwise, recalling that $v$ is the minimal integer such that $x_v \ge 2y$, then $2y-2k \le x_{v-1}<2y$. In addition, $x_{v-1}$ is even, so $x_{v-1}=2y-2j$ for some $j \in \{1,\ldots,k\}$. It follows that 
$$
\varphi(x_{v-1})=x_{v-1}  \, {\prod_{p\;\! \mid \;\! x_{v-1}}} \left(1 - \frac{1}{p}\right) = (y-j) {\prod_{p\;\! \mid \;\! x_{v-1}}}^{\!\!\!\!\prime} \left(1 - \frac{1}{p}\right),
$$
where the last product is extended over the odd prime divisors of $x_{v-1}$. Since by construction we have $y\equiv j\bmod{q_j}$, we obtain by \eqref{eq:doublebounds} that
$$
\varphi(x_{v-1}) \le (y-j) \prod_{i=r_j}^{r_{j+1} - 1} \left(1 - \frac{1}{p_i}\right)< \frac{y-j}{2} < y-k.
$$
Note that the last inequality holds because $y>2k$, thanks to \eqref{boundy}. 

The same argument can be repeated for $x_{v-2}$. 
\end{proof}

We conclude by Claim \ref{claim:upperbound} that
$$
2y \le x_v = \varphi(x_{v-1}) + \varphi(x_{v-2}) + k < 2(y-k)+k \le 2y,
$$
which is contradiction. 
It follows that $x_n<2y$ for all $n \in \mathbb{N}$. 

To complete the proof, it will be enough to show that $2y\le 4^{X^{3^{k+1}}}$. For this, define $X_n:=X^{3^n}$ for all $n\ge 0$ and note that, thanks to Corollary \ref{cor:boundprimes}, we have 
$$
\prod_{X_n<p\le X_n^3}\left(1-\frac{1}{p}\right)<\frac{1}{2} 
\quad \text{ for all }n\ge 0.
$$
By the definition of $r_j$, it follows that 
$$
r_j \le X_j \quad \text{ for all }j \in \{0,1,\ldots,k+1\}.
$$
Lastly, since $\prod_{p\le x}p<4^x$ for all $x\ge 1$ (see e.g. \cite[Lemma 2.8]{MR2919246}), we conclude that 
$$
2y \le 2\prod_{j=0}^k q_j
\le \prod_{i=1}^{r_{k+1}}p_i 
\le \prod_{p\le X_{k+1}}p\le 4^{X_{k+1}}.
$$
Therefore $x_n \le 4^{X_{k+1}}$ for all $n \in \mathbb{N}$. 
\end{proof}


\providecommand{\href}[2]{#2}

\end{document}